\documentclass[11pt]{amsart}
\usepackage{geometry}                
\usepackage{graphicx}
\usepackage{amssymb}
\usepackage{epstopdf}
\newtheorem{theorem}{Theorem}[section]
\newtheorem{lemma}{Lemma}[section]
\newtheorem{corollary}{Corollary}[section]

\newtheorem{conjecture}{Conjecture}[section]
\newtheorem{definition}{Definition}[section]
\newtheorem{example}{Example}[section]
\newtheorem{remark}{Remark}[section]

\numberwithin{equation}{section}

\def\R{\mathbb R}

\def\E{\mathbb E}

\def\G{\Gamma}

\def\om{\omega}
\def\d{\partial}
\def\s{\sigma}
\def\e{\epsilon}
\def\a{\alpha}
\def\b{\beta}
\def\g{\gamma}
\def\l{\lambda}

\title{Detecting intrinsic {global} geometry of an obstacle via  layered scattering}

\author[L.~Bunimovich]{Leonid Bunimovich}
\address{Georgia Institute of Technology, School of Mathematics, 
North Avenue, Atlanta, GA 30332, U.S.A.} 
\email{bunimovh@math.gatech.edu}
 
\author[G.~Katz]{Gabriel Katz}
\address{MIT, Department of Mathematics, 77 Massachusetts Ave., Cambridge, MA 02139, U.S.A.}
\email{gabkatz@gmail.com}

\begin{document}
\maketitle 

\begin{abstract} Given a closed $k$-dimensional submanifold $K$, incapsulated in a compact domain $M \subset \mathbb E^n$, $k \leq n-2$, we consider the problem of determining the intrinsic geometry of the obstacle $K$  (like volume, integral curvature) from the scattering data, produced by the reflections of geodesic  trajectories from the boundary of a tubular $\epsilon$-neighborhood $\mathsf T(K, \epsilon)$ of $K$ in $M$. The geodesics that participate in this scattering emanate from the boundary $\partial M$ and terminate there after a few reflections from the boundary $\partial  \mathsf T(K, \epsilon)$. However, the major problem in this setting is that a ray (a billiard trajectory) may get stuck in the vicinity of $K$ by entering some trap there so that this ray will have infinitely many reflections from $\partial \mathsf T(K, \epsilon)$. To rule out such a possibility, we modify the geometry of a tube $\mathsf T(K, \epsilon)$ by building it from spherical bubbles. We need to use $\lceil \dim(K)/2\rceil$ many bubbling tubes $\{\mathsf T(K, \epsilon_j)\}_j$ for detecting certain global invariants of $K$, invariants which reflect its intrinsic geometry. Thus the words "layered scattering" in the title. These invariants were studied by Hermann Weyl in his classical theory of tubes $\mathsf T(K, \epsilon)$ and their volumes.
\end{abstract}


\section{Introduction}
This paper is about reconstructing a shape of a scatterer  from scattering data. We place the scatterer within a container and send signals (rays) from its boundary to the interior of the container and see whether a signal returns to the boundary.
 We outline a novel approach to the problem of determining  {an} inner global geometry (like volume, curvature, etc.) of an 
 scatterer (obstacle) $K$, residing in the interior of a 
Riemannian manifold $(M, g)$ with boundary, by the probing geodesic rays that originate in the boundary $\d M$.  In the previous  research of this sort, $K$ was a submanifold of $M$ of codimension zero \cite{S}, \cite{St}, \cite{NS}, \cite{NS1}, \cite{NS2}.  We aim here to investigate manifold obstacles $K$ of an arbitrary codimension. 
A goal of this paper is to point out to an interesting interplay between intrinsic and extrinsic geometry of the obstacle $K$. A mathematical setting which we employ here is actually a crude approximation of more realistic physical models which are used in echolocation, linear optics, etc. 
\smallskip


Let $K$ be a compact path-connected subspace of a metric space $M$ equipped with a distance function $d_M: M\times M \to \R_+$. Then $d_M$ induces a new distance function $d_K: K \times K \to \R_+$ on $K$, defined informally as the length $d_K(x, y)$ of a shortest path, \emph{residing in} $K$, between a pair of points $x, y \in K$. The $K$-related geometric entities and quantities that may be expressed only in terms of $d_K$ are called {\sf intrinsic}, while the ones that rely in an essential way on the knowledge of $M$ and $d_M$ are called {\sf extrinsic}. Thus an extrinsic property depends on the \emph{relation} between $K$ and the ambient $M$. For example, if $K$ is a compact connected smooth curve in the plane $M = \mathbb E^2$, then its length is an intrinsic invariant of $K \subset \mathbb E^2$, but the diameter of $K$ is an extrinsic quantity. The common philosophy in mathematics (and science in general) is that the intrinsic properties of subspaces $K \subset M$ are more fundamental then their extrinsic ones. 
\smallskip 

Clearly, the space geodesics in the ambient manifold $M$ that intersect an obstacle $K$ of codimension $\geq 2$ has measure zero; so  the ``direct" scattering from $K$ is physically unobservable. However, thickening the obstacle $K$ may solve the problem.  A manifold obstacle $K$ is a core of a tubular $\e$-neighborhood $\mathsf T(K, \e)$ of $K$, residing in the interior of $M$, where $\e > 0$ is so small that $\mathsf T(K, \e)$ is the space of a fibration over $K$ with the fiber being an $\e$-ball of dimension $\dim M - \dim K$ (see Fig.\ref{TUBES}). The probing geodesic rays generate {\sf scattering data}. The ones, which are of interest to us, are produced by the reflections of geodesic trajectories in $M \setminus \mathsf T(K, \e)$ from the boundary $\d \mathsf T(K, \e)$, where $K$ is assumed to be a closed submanifold of the interior of $M$.  In fact, we need to use \emph{several} nested tubes $\{\mathsf T(K, \e_j)\}_j$ for detecting certain {\sf intrinsic} invariants of $K$, studied by Hermann Weyl in terms of the volumes of the tubes, viewed as functions of $\e$ \cite{We}. Thus the words "layered scattering" appear in the title. The minimal number of such reflective layers is $\lceil k/2\rceil$, where $k =\dim(K)$, the dimension of $K$. One may think of the layered scattering as probing $K$ by waves with sufficiently high frequencies, so that the corresponding wave lengths were much smaller that the sickness $\e$ of a tube. Waves of a frequency $f_j$ are reflected from the  boundary $\d \mathsf  T(K, \e_j)$ of a tube $\mathsf  T(K, \e_j)$. Our ability to determine global inner invariants of $K$ precisely depends on a priori assumptions that the {\sf shadow volumes} $\mathcal W(\e)$ (see formula (\ref{eq.shadow_vol})) of sufficiently narrow tubes $\d \mathsf T(K, \e)$ with the core $K$ vanish. \smallskip

Throughout the paper, we assume that the metric $g$ on $M$ is {\sf non-trapping}, that is, any geodesic in $M$ originates and terminates at points of the boundary $\d M$.  For example, if $M$ is a compact domain in the Euclidean space $\E^n$ or Hyperbolic space $\mathbb H^n$, the metric on $M$ is non-trapping (see \cite{K2}, \cite{K5}). \smallskip

Let an obstacle $K \subset \mathsf{int}(M)$ be a smooth $k$-dimensional submanifold of $M$, not necessarily connected. We denote by $n$ the dimension of $M$ and assume that $n \geq k+2$.


For a sufficiently small $\e >0$, the tube $\mathsf T(K, \e)$ is a disjoint union of all normal to $K$ geodesic $\e$-balls. Each ball is formed by geodesic segments $\g$ of length $\e$, emanating from a point  $m \in K$ so that $\g$ are orthogonal to $K$ at $m$. For a closed manifold $K$ and a small $\e >0$,  the tube $\mathsf T(K, \e)$ coincides with the $\e$-neighborhood $\mathsf U(K, \e)$ of $K$ in $M$. For a manifold $K$ with a boundary, the tube $\mathsf T(K, \e)$ is contained in $\mathsf U(K, \e)$, but has a different from $\mathsf U(K, \e)$ structure in the vicinity of $\d K$.\smallskip

Let $N_\e =_{\mathsf{def}}\, M \setminus  \mathsf{int}(\mathsf T(K, \e))$. Then $\d N_\e = \d M \coprod \d \mathsf T(K, \e)$ for all sufficiently small $\e >0$. 

Denote by $SN_\e$ and $SM$ the spaces of the unitary tangent spherical fibrations over $N_\e$ and $M$, respectively. We consider the the Lioville $1$-form $\b$ and the symplectic $2$-form $\om  = d\b$ on the cotangent bundle $T^\ast M$.  The metric $g$ induces a bundle isomorphism $\Phi_g: T_\ast M \to T^\ast M$. Consider now the pull-backs $\b_g := \Phi_g^\ast(\b)$ and $\om_g = \Phi_g^\ast(\om)$ of these differential forms to $T_\ast M$ and restrict them to $SM$. The form $\b_g \wedge \om^{n-1}_g$, where $n = \dim M$, is a volume form on $SM$.  Here $\om^{n-1}_g$ denotes the $(n-1)$-st exterior power of the symplectic $2$-form $\om_g$. Recall that this volume form $\b_g \wedge \om^{n-1}_g$ is invariant under the geodesic flow $\phi^t_g: SM \to SM$.
\smallskip

We denote by $\d_1^+(SM) \subset \d(SM)$ the set of points $x =(m, v)$, where the tangent vector $v$ is either directed inside of $M$ or is tangent to $\d M$, and $m \in \d M$. Similarly, let $\d_1^-(SM) \subset \d(SM)$ be the set  $\{x =(m, v)\}$, where $v \in T_mM$ points outside of $M$ or is tangent to $\d M$. \smallskip

Consider now the geodesic billiard trajectories $\g$ in the smooth (or at least $C^3$-differen-tiable) Riemannian manifold $N_\e$ that emanate from $\d M$ and are {\sf reflected} from $\d \mathsf T(K, \e)$ according the laws of Geometric Optics \cite{Tab}. We stress that, by our convention, if such a billiard trajectory reaches $\d M$ in positive time, then it terminates there, thus providing us with the ``{\sf scattering data}" as in (\ref{eq.skattering_map}) below. Since the metric $g$ in $M$ is non-trapping, any geodesic curve in $M$ originates and terminates at points of $\d M$. However, some of the billiard trajectories $\g$ that originate in $\d M$ may be ``trapped" in $N_\e \setminus \d M$ 
without reaching $\d M$ again in a positive time. Such a trajectory $\g$ must have infinitely many reflections from the boundary $\d \mathsf T(K, \e)$ and such phenomenon may occur, e.g., if the boundary of K is not smooth enough (see \cite{H}). If an initial point $(m, v) \in \d_1^+(SM)$ generates such a trapped trajectory, we call $(m, v)$ {\sf trapping initial data}. They form a subset $\mathsf{Trap}(\d M\leadsto, SN_\e)$ of $\d_1^+(SM)$. 
Since the metric $g$ in $M$ is non-trapping, any billiard trajectory $\g$ in $N_\e$, which is determined by a point $$x \in  \d_1^+(SM) \setminus  \mathsf{Trap}(\d M \leadsto, SN_\e),$$ will reach $\d M$ again in a finite time. Its lift $\tilde\g$ to $SN_\e$, after a few reflections of the boundary $SN_\e|_{\d \mathsf T(K, \e)}$, will terminate at a point of $\d_1^-(SM)$ (see Fig.\ref{bill}, the top diagram). A smooth involution $\tau: \d_1^-(S\mathsf T(K, \e)) \to \d_1^+(S\mathsf T(K, \e))$ represents the effect of an elastic reflection from the boundary $\d \mathsf T(K, \e)$.

As a result, we have the following {\sf scattering map} 
\begin{eqnarray}\label{eq.skattering_map}
\mathcal B(K, \e): \d_1^+(SM) \setminus  \mathsf{Trap}(\d M \leadsto, SN_\e) \to \d_1^-(SM),
\end{eqnarray}
which is our main probing instrument (see Fig.\ref{bill} and Fig.\ref{fig.reflection}).  We will adjust $\mathcal B(K, \e)$ by varying $\e > 0$.
The map $\mathcal B(K, \e)$ preserves the $(2n-2)$-dimensional measure $\mu$ on $\d(SN_\e)$, defined by integrating the differential form $\pm\om_g^{n-1}|_{\d(SN_\e)}$ against Lesbegue measurable sets in $\d_1^\pm(SN_\e)$ \cite{K5}. 
\smallskip
 
Our main results about the recovery of the intrinsic global invariants of an obstacle $K$ from the scattering data, delivered by the maps $\{\mathcal B(K, \e)\}_\e$,  are described in Theorem \ref{th.main}, Theorem \ref{th.traping_volume}, Corollary \ref{cor.av_length}, and Theorem \ref{th.main_bubbling}.\smallskip

\begin{figure}[ht]
\centerline{\includegraphics[height=4in,width=4.3in]{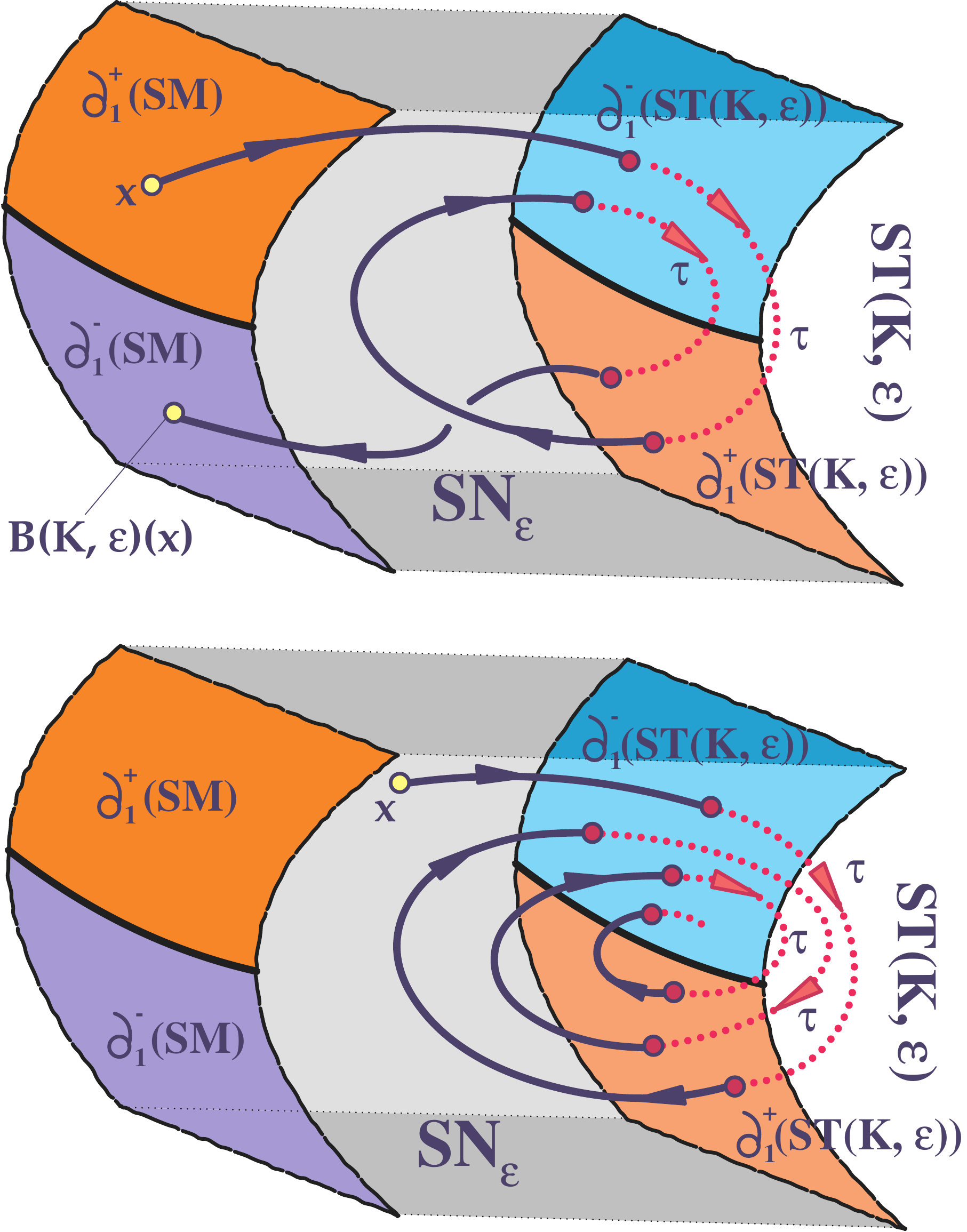}}\label{bill}
\bigskip
\caption{\small{{\sf Top picture:} the lift $\tilde\g$ of a non-trapped billiard trajectory $\g$ in $N_\e$ to $SN_\e$, which originates at $x \in \d_1^+(SM)$ and, after two ``reflections" $\tau$ from $\d_1^-(\mathsf T(K, \e))$, terminates at a point $\mathcal B(K, \e)(x)\in \d_1^-(SM)$. {\sf Bottom picture}: the lift $\tilde\g$ to $SN_\e$ of a trapped billiard trajectory $\g$ in $N_\e$, which originates at $x \in SN_\e$ and does not reach $\d(SM)$ again.}} 
\end{figure}

In summary, a physical model, which we explore here, is the one used in echolocation or in linear optics, where a probing signal originates at the boundary of a known region $M$, and the unknown obstruction $K \subset \mathsf{int}(M)$ is a submanifold, surrounded by several reflecting layers (e.g., with different reflecting properties with respect to different penetrating abilities). 

\section{Probing an obstacle by remote skattering}

In what follows, we use the word "{\sf locus}" as a synonym for "set of points that satisfies a given geometric property".
The locus $\mathsf{Trap}(\d M \leadsto,\; SN_\e)$ has the Lesbegue measure zero (\cite{LP}, Theorem 1.6.2).  In contrast, {\sf the trapping locus} $\mathsf{Trap}(SN_\e) \subset SN_\e$, that consists of such initial data $(m, v) \in SN_\e$ for which the billiard trajectory in $N_\e$ will never reach $\d M$, may have a \emph{positive} Lesbegue measure $\mu$ \cite{Pe} (see Fig.1, the bottom diagram)! This possibility significantly complicates our efforts. We conjecture though that, for a non-trapping metric $g$ on $M$ and a generic obstacle $K$ whose dimension is less than $n-1$, the locus $\mathsf{Trap}(SN_\e)$ has zero measure $\mu$ for all sufficiently small $\e > 0$ (see Conjecture \ref{conjecture A}). 

It is known that, if $M$ is a ball in the Euclidean space $\E^n$ and the tube $\mathsf T(K, \e)$ is replaced by any disjoint union $\tilde K$ of {\sf convex} domains, then $\mu(\mathsf{Trap}(\d M \leadsto,\; \tilde K)) = 0$ by \cite{NS}, \cite{NS2}. (Note that the intrinsic geometry and topology of the manifolds  $\tilde K$'s that are unions of convex domains is trivial...) 
Therefore, in the case of $K$ being a finite set in $\E^n$, one has $\mu(\mathsf{Trap}(SN_\e)) = 0$. On the other hand, if $K$ is not of a homotopy type of a finite set, then the tube $\mathsf T(K, \e)$ cannot be a disjoint union of convex domains for all sufficiently small $\e > 0$.
\smallskip

We denote by $L_\e(x)$ the length of the billiard trajectory (which may reflect from $\d \mathsf  T(K, \e)$ several times) in $N_\e$, determined by the initial data $x \in \d_1^+(SM) \setminus  \mathsf{Trap}(\d M \leadsto,\; SN_\e)$. It is known that $L(x) < \infty$ almost everywhere in $x \in \d_1^+(SM)$ \cite{LP}. 
\smallskip

Our aim is to probe $K$ by means of the scattering maps $\{\mathcal B(K, \e)\}_\e$ (see (\ref{eq.skattering_map})). In order to detect some global intrinsic invariants of the Riemannian manifold $K$ (which will be described in the next section), we will use a nested family of sufficiently narrow tubes $$\mathsf T(K, \e_0) \supset \mathsf T(K, \e_1) \supset \ldots \supset \mathsf T(K, \e_k)$$  
and the family of the scattering maps $\{B(K, \e_i)\}_{0 \leq i \leq k}$ that are associated with them. For a ``clean" detection, we need to assume that $\mu(\mathsf{Trap}(SN_{\e_i}) = 0$. It worthwhile to mention that we do not require $\mathsf T(K, \e_i)$  to be narrow. \smallskip

\begin{figure}[ht]
\centerline{\includegraphics[height=3.5in,width=6in]{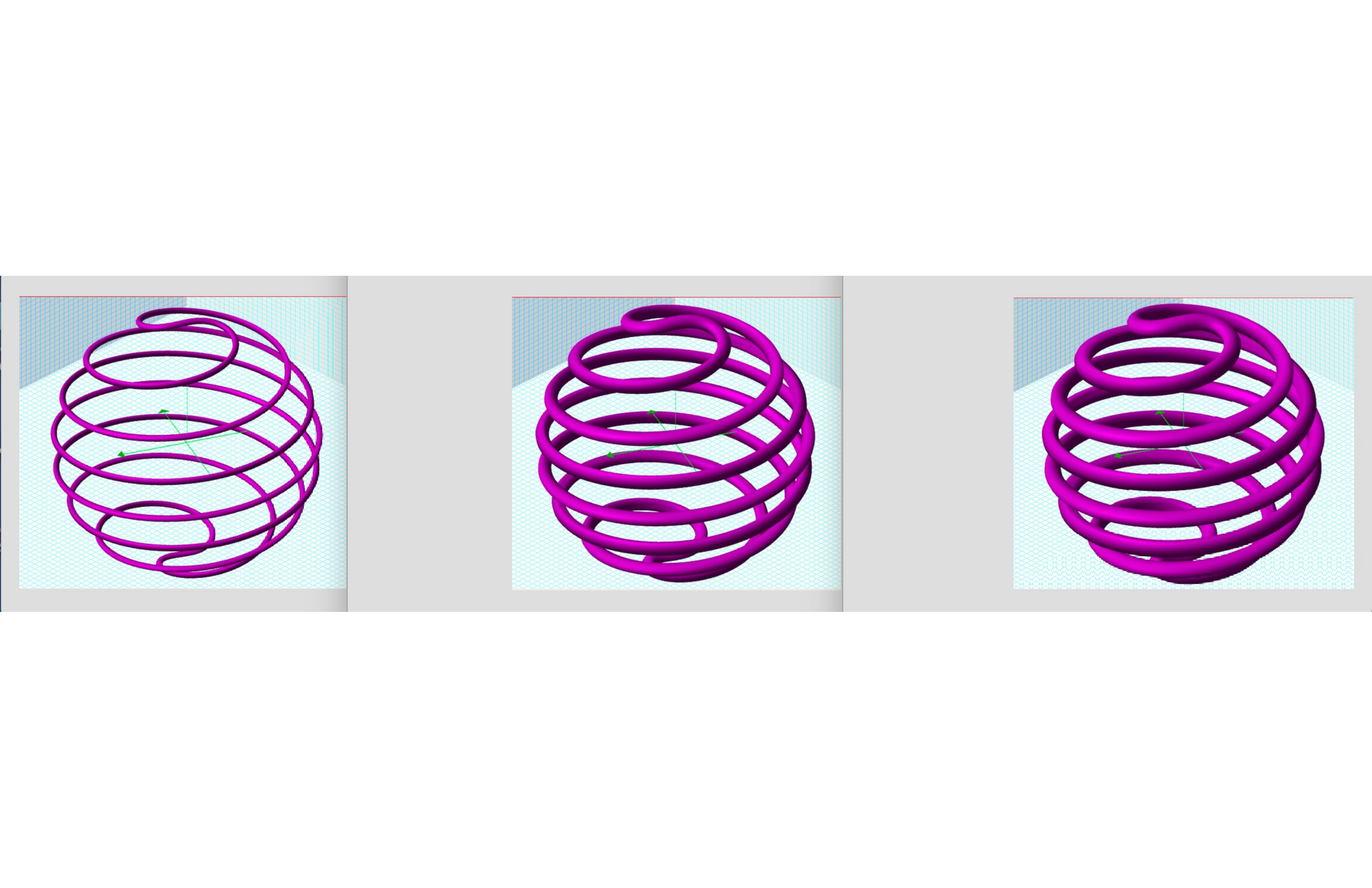}}
\bigskip
\caption{\small{Three Weyl tubes $\{\mathsf T(K,\e_i)\}_{i =1, 2, 3}$ with the same $1$-dimensional core loop $K \subset \E^3$.}} \label{TUBES}
\end{figure}

\section{Weyl's theory of tubes}

To formulate the Hermann Weyl theory of tubes \cite{We}, we need to introduce several geometric ingredients.\smallskip 

Let  $K$ be a compact smooth $k$-dimensional submanifold of the Euclidean space $\E^n$. We also consider $K$ as a Riemannian manifold with a symmetric connection $\{\Gamma^\kappa_{\; \a\b}\}$. \smallskip

Following \cite{We}, we introduce the {\sf Riemann curvature tensor} $\{R^\kappa_{\; \l\a\b}\}$ on $K$ as:
\begin{eqnarray}
R^\kappa_{\; \l\a\b} = \Big(\frac{\d\G^\kappa_{\,\l\b}}{\d u^\a} - \frac{\d\G^\kappa_{\,\l\a}}{\d u^\b} \Big) +
\sum_\rho \big(\G^\kappa_{\, \rho\a}\; \G^\rho_{\, \l\b} - \G^\kappa_{\, \rho\b}\; \G^\rho_{\, \l\a}\big),
\end{eqnarray}
and the skew-symmetric in $\a, \b$ and in $\kappa, \l$ the {\sf Weyl tensor} $\{R^{\kappa\mu}_{\a\b}\}$ by the formula:

\begin{eqnarray}
R^\kappa_{\; \l\a\b} = \sum_\mu g_{\l\mu} R^{\kappa\mu}_{\a\b}
\end{eqnarray}

\begin{theorem}\label{We}\cite{We} For any compact smooth $k$-dimensional submanifold $K \subset \E^n$,
the volume of the tube $\mathsf T(K, \e)$, for all sufficiently small $\e > 0$, is a polynomial $P_K(\e)$ of degree $n$ in the variable $\e$, divisible by the monomial $\e^{n-k}$. It is given by the formula
\begin{eqnarray}\label{a}
& vol(\mathsf T(K, \e)) =  \\ 
& vol(B^{n-k})\cdot \sum_{\{\ell\, \in\, [0, k], \; \ell \equiv 0 \mod 2\}}\, \frac{\mathsf Q_\ell(K)}{(n-k+2)(n-k+4) \ldots (n-k +\ell)}\; \e^{n-k+\ell},\nonumber
\end{eqnarray}
where $vol(B^{n-k})$ is the volume of the unit Euclidean ball, and 
\begin{eqnarray}\label{b}
\mathsf Q_\ell(K) := \int_K H_\ell \; \mu_K,
\end{eqnarray}
where the measure $\mu_K$ is the volume $k$-form on $K$,  and the function $H_\ell: K \to \R$ can be written in terms the Weyl tensor $\{R^{\kappa\mu}_{\a\b}\}$ of $K$ as  
\begin{eqnarray}\label{c}
H_\ell :=
\frac{1}{2^\ell (\ell/2)!} \sum_{1',\, \ldots\, ,\, \ell'}\Big\{\pm \sum_{\a_1, \dots , \a_\ell} R^{\a_{1'}\a_{2'}}_{\a_1\a_2}\; R^{\a_{3'}\a_{4'}}_{\a_3\a_4}\; \ldots \; R^{\a_{\ell'-1}\a_{\ell'}}_{\a_{\ell-1}\a_\ell}\Big\}. 
\end{eqnarray} 
\hfill $\diamondsuit$
 \end{theorem}
 
The first coefficient $\mathsf Q_0(K)$ in (\ref{a}) is just the $k$-volume of $K$, while the second coefficient $\mathsf Q_2(K) = \frac{1}{2}\int_K \kappa\, d\mu$, where $\kappa$ is the scalar curvature of $K$ \cite{We}. \smallskip

By their very construction, all the terms $\{\mathsf Q_\ell(K)\}_\ell$ do not depend on the {\sf extrinsic geometry} of the embedding $K \subset \E^n$; i.e.,  $\{\mathsf Q_\ell(K)\}_\ell$ are the same for all {\sf isometric} embeddings of a given Riemannian manifold $K$ into the Euclidean space. In fact, these coefficients can be expressed in terms of the second fundamental form on $K$ \cite{We}. In case when $K$ is closed and $\dim K = 2q$, the leading coefficient $\mathsf Q_{2q}(K) = (2\pi)^q \chi(K)$, where $\chi(K)$ is the Euler characteristic of $K$ \cite{Gr}. For a compact surface $K \subset \R^3$, the formula (\ref{a}) reduces to the following cubic polynomial in $\e$
$$vol(\mathsf T(K, \e)) = 2\e \cdot vol(K) + \frac{4\pi \e^3}{3} \cdot \chi(K),$$
which is divisible by $\e$.
 
See \cite{Gr} for a comprehensive account of the analogues of these formulas, when the $\e$-tubes $\mathsf T(K, \e)$ are considered in a general Riemannian manifold $(M, g)$.


\section{Layered scattering with reflections from the nested Weyl tubes}

We introduce the following notations: $$vol(SM):= \int_{SM} \b_g\wedge \om_g^{n-1},\; vol(SN_\e):= \int_{SN_\e} \b_g\wedge \om_g^{n-1}, \; vol(S\mathsf T(K,\e)):= \int_{S\mathsf T(K,\e)} \b_g\wedge \om_g^{n-1}.$$

Let us formulate now an important claim from \cite{St}, as it applies to our setting:
\begin{theorem}\label{St}{\cite{St}} Let $M$ be a compact smooth $n$-dimensional manifold, equipped with a non-trapping metric $g$, and  $N_\e := M \setminus \mathsf T(K, \e)$, where $\e > 0$ is so small that the $\e$-tubes, whose cores are the connected components of $K$,  do not intersect each other and $\mathsf T(K,\e) \cap \d M = \emptyset$. 

Then the volume of the non-trapping region in $SN_\e$ can be calculated as
\begin{eqnarray}\label{eq1.1}
\int_{\big\{ SN_\e \setminus \mathsf{Trap}(SN_\e)\big\}} \b_g\wedge \om_g^{n-1} =  
\int_{\big\{\d_1^+(SM)\, \setminus\, \mathsf{Trap}(\d M\leadsto,\; SN_\e)\big\}} 
L_\e\cdot\om_g^{n-1}.
\end{eqnarray}
Here $L_\e(x)$ is the combined length of several geodesics flow $\phi^t_g: SN_\e \to SN_\e$ trajectories, the first of which starts at a point $x \in \d_1^+(SM) \setminus \mathsf{Trap}(\d M\leadsto,\, SN_\e)$ and the last of which terminates at a point $y \in \d_1^-(SM)$; the other ends of these trajectories reside in $\d(S\mathsf T(K, \e)) \subset \d(SN_\e)$ (see Fig.1, the upper diagram).\footnote{the length of a $v^g$-trajectory $\tilde\g \subset SN_\e$ in the {\sf Sasaki metric} $gg$ (see \cite{Sa}) on $SN_\e$ equals the length of its projection $\g \subset N_\e$ in the metric $g$.} The images of these trajectories, under the projection $\pi: SN_\e \to N_\e$, give rise to the broken billiard trajectory in $N_\e$ that starts at $\pi(x)$ and, after several reflections from $\d\mathsf T(K, \e)$, terminates at $\pi(x)$.  \hfill $\diamondsuit$
\end{theorem}

The relation in (\ref{eq1.1}) uses that $\b_g(v^g) =1$, where $v^g$ is {\sf the geodesic vector field} on $SN_\e$.\smallskip

Let
\begin{eqnarray}\label{V} 
\mathcal V(\e) := \int_{\big\{ SN_\e \setminus \mathsf{Trap}(SN_\e)\big\}} \b_g\wedge \om_g^{n-1}.
\end{eqnarray} 

Stoyanov's theorem says that the volume $\mathcal V(\e)$ is ``observable" and may be computed from the scattering information, encoded in the following integrable function 
\begin{eqnarray}
L_\e: \d_1^+(SM) \setminus \mathsf{Trap}(\d M\leadsto,\, SN_\e) \to \R_+.
\end{eqnarray} 
In other words, a value $\mathcal V(\e)$ is detected by the scattering data that include a {\sf traveling time} of the probing signal (of an appropriate frequency $f_\e$). Note that any two functions $L_\e$ and $\tilde L_\e$, which differ on a set of measure zero, both make equal contributions to the integral in the LHS of (\ref{eq1.1}). \smallskip

We will call a volume of the trapping portion $\mathsf{Trap}(SN_\e)$ of $SN_\e$ a {\sf shadow $\e$-volume} of $K$. It is given by the formula 
\begin{eqnarray}\label{eq.shadow_vol}
\mathcal W(\e) := vol(\mathsf{Trap}(SN_\e)) = \int_{\big\{\mathsf{Trap}(SN_\e)\big\}} \b_g\wedge \om_g^{n-1}
\end{eqnarray} 
At the first glance,  $\mathcal W(\e)$ looks as an {\sf extrinsic} invariant of the obstacle $K$, which, among other things, depends on $M$ and the inclusion $K \hookrightarrow M$. \smallskip

By definition,
\begin{eqnarray}\label{eq.volumes}
vol(SM) = vol(SN_\e) + vol(S\mathsf T(K,\e)) = \mathcal V(\e) + \mathcal W(\e) + vol(S\mathsf T(K, \e)). 
\end{eqnarray}

Thus, by Theorem \ref{St}, the shadow $\e$-volume equals
\begin{eqnarray}\label{eq.VOLUMES}
 \mathcal W(\e) \; = \; vol(SM) - vol(S\mathsf T(K, \e)) -  \int_{\big\{\d_1^+(SM) \setminus \mathsf{Trap}(\d M\leadsto,\, SN_\e)\big\}} L_\e \cdot \om_g^{n-1}. \nonumber
\end{eqnarray}
A volume $vol(S\mathsf T(K, \e))$ will turn out to be an invariant that is shared by all {\sf isometric} embeddings $\a: K \hookrightarrow M$, subject to the constraints $\mathsf{dist}_g(\d M, \a(K)) \geq \e$. The integral $\mathcal V(\e)$ in (\ref{V})  is an {\sf observable} quantity. \smallskip

Generally, the shadow $\e$-volume $\mathcal W(\e)$ is ``a black hole" or ``a known unknown" of the whole scattering enterprise. However, it has some good features too: i.e., by Theorem 1.2 from \cite{St}, $\mathcal W(\e)$ is a continuous function of the smooth embedding $\d \mathsf T(K, \e) \hookrightarrow M$, and thus of the parameter $\e >0$ and of the smooth regular embedding $\a: K  \hookrightarrow M$.
\smallskip

\begin{figure}[ht]
\centerline{\includegraphics[height=3in,width=3.5in]{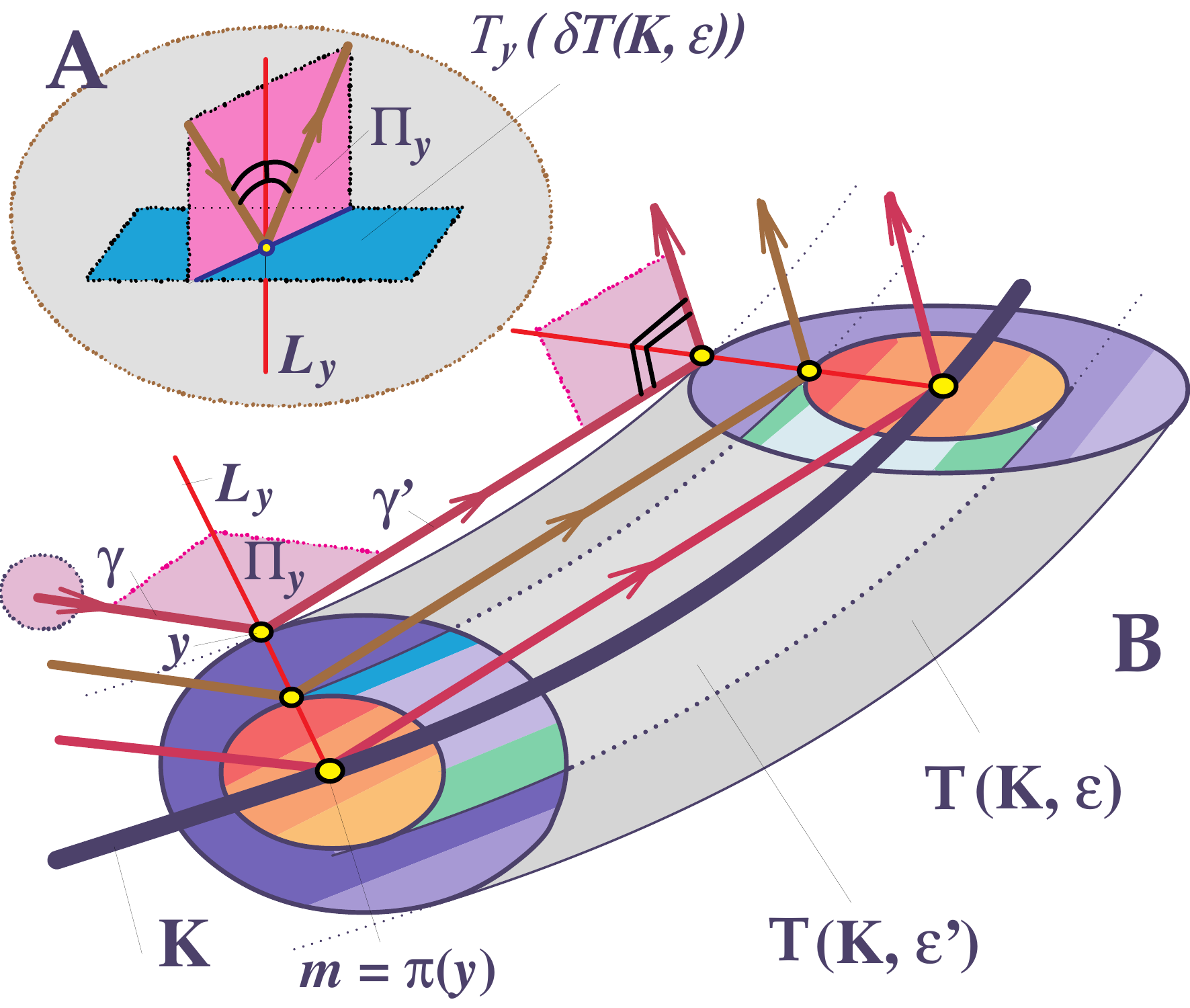}}\label{BILL}
\bigskip
\caption{\small{{\sf Diagram A :} A single reflection from the boundary of the tube $\mathsf T(K, \e)$. {\sf Diagram B :} Double reflections from the boundaries of two tubes, $\mathsf T(K, \e)$ and $\mathsf T(K, \e')$, and from their core $K$.}} \label{fig.reflection}
\end{figure}

$\bullet$ {\sf \underline{Hypothesis A}} {\it For all sufficiently small $\e> 0$, the volume $\mathcal W(\e)$ of the trapping locus $\mathsf{Trap}(SN_\e)$ is zero.} \hfill $\diamondsuit$ \smallskip

\begin{conjecture}\label{conjecture A} Hypothesis A is valid for any regular imbedding $K \subset \E^n$ of a smooth manifold $K$ of codimension $\geq 2$. \hfill $\diamondsuit$
\end{conjecture}
 
Our main result is obtained by combining Theorem \ref{We} and Theorem \ref{St}:
 
\begin{theorem}\label{th.main} Let $K$ be a compact smooth $k$-dimensional submanifold of the Euclidean space $\E^n$. Let $\mathsf T(K, \e)$ be a $\e$-tube, properly contained in a smooth compact domain $M \subset \E^n$, and denote $N_\e := M \setminus \mathsf{int}(\mathsf T(K, \e))$. 
 
If  Hypothesis A is valid, then, for any sequence $\e_0 > \e_1 > \dots > \e_{\lceil k/2\rceil} > 0$, where $\e_0$ is sufficiently small, 
the lengths (travel times) functions  $\big\{L_{\e_j}: \d_1^+(SM) \to \R\big\}_{0 \leq j \leq \lceil k/2\rceil}$ of billiard trajectories in the domains $\{N_{\e_j}\}_{0 \leq j \leq \lceil k/2\rceil}$ determine the  intrinsic global invariants $$\Big\{\mathsf Q_\ell := \int_K H_\ell \; \mu_K\Big\}_{0 \leq \ell \leq \lceil k/2 \rceil}$$ of the Riemannian manifold $K$, defined in formulae (\ref{b})-(\ref{c}). 
\end{theorem}

\begin{proof} 
Let $\e := \e_i$. 
By equation (\ref{eq.volumes}) and the Fubini Theorem, 
$$vol(\mathsf{Trap}(SN_\e)) = vol(SN_\e) -\mathcal V(\e) = vol(SM) - vol(S\mathsf T(K,\e)) - \mathcal V(\e)$$
$$\stackrel{Fubini}{=} \; \s_{n-1} \cdot \big[vol(M) - vol(\mathsf T(K,\e))\big] - \mathcal V(\e),$$
where $\s_{n-1}$ denotes the volume of the standard $(n-1)$-sphere of radius $1$ in $\E^n$.  \smallskip

By (\ref{a}), $vol(T(K,\e))$ is given by a polynomial $P_K(\e)$ of the form $\e^{n-k}\cdot Q_K(\e^2)$, where $Q_K(\e^2)$ is polynomial of degree $\lceil k/2 \rceil$ in $\e^2$.
Therefore 
\begin{eqnarray}\label{eq.d}
vol(\mathsf{Trap}(SN_\e)) = \s_{n-1} \cdot \big[vol(M) - \e^{n-k}\cdot Q_K(\e^2))\big] - \mathcal V(\e).
\end{eqnarray}

Assuming that $vol(\mathsf{Trap}(SN_\e)) = 0$, we get 
\begin{eqnarray}\label{e}
Q_K(\e^2)) =  \frac{1}{\e^{n-k}}\Big[ vol(M) - \frac{1}{\s_{n-1}} \mathcal V(\e)\Big], 
\end{eqnarray}
where $\mathcal V(\e)$ is defined by formula (\ref{V}) and is determined by an integrable function \hfill\break $L_\e: \d_1^+(SM) \setminus \big(\mathsf{Trap}(SN_\e, \d M) \to \R$.

Since $Q_K$ is a polynomial of degree $\lceil k/2 \rceil$ in $\e^2$, then choosing any sequence $ \e_0 > \e_1 > \ldots > \e_{\lceil k/2 \rceil} > 0$, and applying (\ref{e}) to the members of the sequence,  we conclude that the collection of integrable functions $\{L_{\e_i}: \d_1^+(SM) \setminus \mathsf{Trap}(\d M \leadsto\, SN_{\e_i}) \to \R\}_{0 \leq i \leq \lceil k/2 \rceil}$ determines the polynomial $Q_K$. Thus the polynomial $P_K = \e^{n-k}\cdot Q_K$ is  the one in (\ref{a}). As a result, the intrinsic integral invariants $\{\mathsf Q_\ell(K)\}$ from formula (\ref{b}) are determined by the scattering data $\{L_{\e_i}\}_{0 \leq i \leq \lceil k/2 \rceil}$, provided that $\{vol(\mathsf{Trap}(SN_{\e_i})) = 0\}_{0 \leq i \leq \lceil k/2 \rceil}$.
\end{proof}

\begin{corollary}\label{cor.Q} Under the Hypotheses A, 
the length/travel time integrable functions \hfill \break  $\big\{L_{\e_j}: \d_1^+(SM) \to \R\big\}_{0 \leq j \leq \lceil k/2\rceil}$ are able to detect the following quantities: 

$(1)$ the $k$-dimensional volume of the obstacle $K$, 

$(2)$ the integral $\mathsf Q_2(K) = \frac{1}{2}\int_K \kappa\, d\mu$, where $\kappa$ is the scalar curvature of $K$, 

$(3)$ for an even $k$, the following invariant $$\mathsf Q_k(K) = (2\pi)^{k/2} \chi(K) = (2\pi)^{k/2} \deg(\mathsf{Gauss}(\d T(K, \e)),$$ 
where $\chi(K)$ is the Euler number of $K$, and $\deg(\mathsf{Gauss}(\d T(K, \e))$ is the degree of the Gaussian map $\mathsf{Gauss}: \d \mathsf T(K, \e) \to S^{n-1}$ (defined by the exterior normals to $\d \mathsf T(K, \e)$).
\end{corollary} 

 \begin{proof} All the claims follow from the geometric interpretations of the coefficients $\{\mathsf Q_\ell(K)\}_\ell$ in \cite{We}. In particular, in the case when $K$ is closed and $\dim K = k = 2q$, the leading coefficient $\mathsf Q_{2q}(K) = (2\pi)^q \chi(K)$, where $\chi(K)$ is the Euler characteristic of $K$ \cite{Gr}. This number coincides with the degree $(2\pi)^{k/2} \deg(\mathsf{Gauss}(\d \mathsf T(K, \e))$ of the Gaussian normal map times $(2\pi)^q$.
 \end{proof}
 
Since the Euler class $\chi(K)$ determines the genus of the closed connected orientable surface $K$,  Corollary \ref{cor.Q} implies the following claim.

\begin{corollary} Let $M$ be a compact domain in $\E^n$ with a smooth boundary, and $K$ is a  closed smooth connected $2$-dimensional surface in the interior of $M$. Then the scattering data $\big\{L_{\e_j}: \d_1^+(SM) \to \R_+\big\}_{0 \leq j \leq 1}$ detect the genus of $K$, provided that the Hypothesis A holds. \hfill $\diamondsuit$
 \end{corollary}
 
 Define now {\sf the average length} (travel time) of  non-trapped billiard trajectories in $N_\e$ which originate and terminate in $\d M$, with possible multiple reflections from $\d T(K, \e)$ in-between,  by
\begin{eqnarray}\label{eq.av_length_A} 
\mathcal L_\e^{\mathsf{av}}(M, K) :=  \frac{\int_{\big\{\d_1^+(SM) \setminus \mathsf{Trap}(\d M\leadsto,\, SN_\e)\big\}} L_\e \cdot \om_g^{n-1}}{\int_{\big\{\d_1^+(SM) \setminus \mathsf{Trap}(\d M\leadsto,\, SN_\e)\big\}} \;\om_g^{n-1}}.
\end{eqnarray}

\begin{theorem}\label{th.traping_volume} Let $K$ be a closed smooth $k$-dimensional Riemanian manifold, and $M$ is a smooth compact domain in the Euclidean space $\E^n$. Then, for any {\sf isometric} imbedding $\a: K  \hookrightarrow M \setminus T_\e(\d M)$, the trapping volume equals $$vol(\mathsf{Trap}(SN_\e)) := \int_{\big\{\mathsf{Trap}(SN_\e)\big\}} \b_g\wedge \om_g^{n-1}.$$ It depends continuously on $\e > 0$, on the average length  $\mathcal L_\e^{\mathsf{av}}(M, \a(K))$ of the non-trapped billiard trajectories in $N_\e(\a) := M \setminus T(\a(K), \e)$, and on the following three $\a$-independent\footnote{The quantity in {\bf (i)} is $K$-dependent, but, by Theorem \ref{We}, it is $\a$-independent.} volumes: \smallskip

{\bf (i)} $\int_{S\mathsf T(\a(K),\, \e)} \;\b_g \wedge \om_g^{n-1}$, {\bf (ii)} $\int_{\d_1^+(SM) } \;\om_g^{n-1}$,  {\bf(iii)} $\int_{SM} \;\b_g \wedge \om_g^{n-1}$. 

Moreover,  all these quantities together determine $vol(\mathsf{Trap}(SN_\e))$.
\end{theorem}

\begin{proof} By  Theorem 1.6.2, \cite{LP}, the locus $\mathsf{Trap}(\d M \leadsto, SN_\e)$ has the Lesbegue measure zero. Thus the denominator of (\ref{eq.av_length_A}) equals $$\int_{\big\{\d_1^+(SM) \setminus \mathsf{Trap}(\d M\leadsto,\, SN_\e)\big\}} \;\om_g^{n-1} = \int_{\big\{\d_1^+(SM) \big\}} \;\om_g^{n-1},$$
and it does not depend on $\a$ and $K$. Therefore the numerator in  (\ref{eq.av_length}) depends only on the average length $\mathcal L_\e(\a(K))$.

On the other hand, by Theorem \ref{We}, the volume of an $\e$-tube $\mathsf T(\a(K), \e)$ depends only on the intrinsic geometry of $K$. Thus, by the Fubini Theorem, the volume of $S\mathsf T(\a(K), \e) \subset SM$ depends only on the intrinsic geometry of $K$ as well. 

Therefore, formula (\ref{eq.volumes}) implies that the trapping volume $vol(\mathsf{Trap}(SN_\e))$ depends only on the following four quantities:\smallskip 

{ \bf (1)} $\mathcal L_\e^{\mathsf{av}}(M, \a(K))$, {\bf (2)} $\int_{S\mathsf T(\a(K), \e)} \;\b_g \wedge \om_g^{n-1}$, {\bf (3)} $\int_{\d_1^+(SM) } \;\om_g^{n-1}$, and {\bf(4)} $\int_{SM} \;\b_g \wedge \om_g^{n-1}$. Here the first quantity is \emph{observable} and \emph{extrinsic} to $K$, the second is \emph{intrinsic} to $K$, and the last two quantities are $\a$- and $K$-\emph{independent}. 

Finally, by \cite{St}, Theorem 1.2, the  volume $vol(\mathsf{Trap}(SN_\e))$ of the trapping region  depends continuously on the embedding $\d \mathsf T(K, \e) \hookrightarrow M$ and, therefore, on the embedding $K  \hookrightarrow M$ and on $\e$, provided that $\e >0$ is sufficiently small.
\end{proof}

\begin{example} \emph{Let $\a: K \hookrightarrow \E^3$ be a link or a knot. Of course, in such cases, $K$ has no interesting intrinsic geometry; but it has its total length $\ell(K)$. Then the polynomial
$$Q_K(\e^2)) =  \frac{1}{\e^{2}}\Big[vol(M) - \frac{1}{\s_{2}} \mathcal V(\e)\Big].$$
 equals to the constant $\pi \cdot \ell(K)$. This leads to 
 a representation of the $\e$-trapping volume as a sum of one observable and two intrinsic  quantities:
 $$\mathcal W(\e) = - \int_{\big\{\d_1^+(SM) \setminus \mathsf{Trap}(\d M\leadsto,\, SN_\e)\big\}} L_\e \cdot \om_g \wedge \om_g - 4\pi \cdot vol(M) - 4\pi^2 \e^2 \cdot\ell(K).$$
 It seems that one can construct an example of a loop $K \subset \E^3$ such that $\mathcal W(\e) > 0$ for {\sf some} ``sufficiently big" $\e > 0$.
 However, if $\mathcal W(\e) =0$ for a small $\e >0$, then $\ell(K)$ is determined by the one-layered scattering as: 
 $$\ell(K) = \frac{1}{4\pi^2 \e^2}\big\{4\pi \cdot vol(M) - \mathcal V(\e)\big\}.$$
 \hfill $\diamondsuit$ }
\end{example}

\begin{corollary} 
Let us adopt the notations of Theorem \ref{th.main}.
In particular, let $M$ be an $n$-dimensional smooth domain in $\E^n$, and $K$ is a compact closed smooth $k$-dimensional submanifold of $M$. 

Then there exists a constant $c(K)>0$, which depends only on the intrinsic geometry of $K$ (in particular, $c(K)$ is $\e$-independent and $M$-independent) and such, that for all sufficiently small  $\e > 0$  the average length/travel time $\mathcal L_\e^{\mathsf{av}}(M, K)$  of billiard trajectories in $N_\e$ (see (\ref{eq.av_length_A})) can detect the volume $$vol(\mathsf{Trap}(SN_\e)) := \int_{\mathsf{Trap}(SN_\e)} \b_g \wedge (\om_g)^{n-1}$$ of the trapping region with the precision $c(K)\cdot \e^{n-k}$.  
 \end{corollary}
 
\begin{proof} The claim follows from the formula (\ref{eq.d}) by bounding from above the absolute value $|Q_K(\e^2)|$, say, in the interval $[0,1]$, provided $\e^2 < 1$. 
Note that $c(K)$ depends only on the inner geometry of $K$ and thus it is shared by all isometric embeddings $\a: K \hookrightarrow M$. 
\end{proof}


The next corollary shows that only the knowledge of the average lengths/travel times $\{\mathcal L^{\mathsf{av}}_{\e_j}(M, K)\}_{0 \leq j \leq \lceil k/2\rceil}$ is needed in order to determine the global Weyl invariants of $K$, provided that the Hypothesis A is valid. 

\begin{corollary}\label{cor.av_length} 
Let $K$ be a closed smooth $k$-dimensional Riemanian manifold and $M$ is a smooth compact domain in the Euclidean space $\E^n$. Then, assuming the Hypothesis A, the average length (travel time) $\mathcal L^{\mathsf{av}}_\e(M, K)$ of the non-trapped billiard trajectories in $N_\e$ which originate and terminate in $\d M$ (with possible multiple reflections from $\d \mathsf T(K, \e)$ in-between) is shared by all {\sf isometric} embeddings 
 $\a: K  \hookrightarrow M \setminus \mathsf T_\e(\d M)$. 
  In fact, 
 \begin{eqnarray}\label{eq.av_length}
 \mathcal L^{\mathsf{av}}_\e(M, K) = \frac{\int_{SM} \;\b_g \wedge \om_g^{n-1} - \int_{S\mathsf T(\a(K),\, \e)} \;\b_g \wedge \om_g^{n-1}}{\int_{\d_1^+(SM) } \;\om_g^{n-1}},
 \end{eqnarray}
 where the  $\a(K)$-dependent quantity $\int_{S\mathsf T(\a(K),\, \e)} \;\b_g \wedge \om_g^{n-1}$ depends in fact only on the intrinsic geometry of $K$.

Moreover, for any sequence $\e_0 > \e_1 > \dots > \e_{\lceil k/2\rceil} > 0$, where $\e_0$ is sufficiently small, the averages of lengths/travel times $\{\mathcal L^{\mathsf{av}}_{\e_j}(M, K)\}_{0 \leq j \leq \lceil k/2\rceil}$ of the non-trapped billiard trajectories in $\{N_{\e_j}\}_{0 \leq j \leq \lceil k/2\rceil}$, which originate and terminate in $\d M$, determine the intrinsic global invariants $$\Big\{\mathsf Q_\ell := \int_K H_\ell \; \mu_K\Big\}_{0 \leq \ell \leq \lceil k/2 \rceil}$$ of the Riemannian manifold $K$, defined by the formulae (\ref{b})-(\ref{c}). 
\end{corollary}

\begin{proof} By Theorem 1.6.2, \cite{LP},  the locus $\mathsf{Trap}(\d M \leadsto, SN_\e)$ has the Lesbegue measure zero. Thus the denominator of (\ref{eq.av_length}) equals $$\int_{\big\{\d_1^+(SM) \setminus \mathsf{Trap}(\d M\leadsto,\, SN_\e)\big\}} \;\om_g^{n-1} = \int_{\big\{\d_1^+(SM) \big\}} \;\om_g^{n-1},$$
and it is a quantity that does not depend on $\a$ and $K$. In fact, by the Fubini Theorem, $\int_{\big\{\d_1^+(SM) \big\}} \;\om_g^{n-1} = \theta_{n-1}\cdot vol(\d M)$, where $\theta_{n-1}$ denotes the volume of the unit $(n-1)$-ball. Therefore, the numerator in  (\ref{eq.av_length}) depends on the average length $\mathcal L^{\mathsf{av}}_\e(\a(K))$ and $vol(\d M)$.

On the other hand, by Theorem \ref{We}, the volume of the $\e$-tube $\mathsf T(\a(K), \e)$ depends only on the intrinsic geometry of $K$. Thus, by the Fubini Theorem, the volume of $S\mathsf T(\a(K), \e) \subset SM$ depends only on the intrinsic geometry of $K$ as well. 

Therefore, the formula (\ref{eq.volumes}) implies (\ref{eq.av_length}). 
Note that  $\mathcal L^{\mathsf{av}}_\e(M, K)$  is an \emph{observable} quantity.
%

The last claim follows from Theorem \ref{th.main}.
\end{proof}


\section{Approximating Weyl tubes by tamed bubbling tubes}

As before, we assume that the obstacle $K \subset M \subset \E^n$ is a closed smooth $k$-dimensional manifold and $M$ is a compact domain with a smooth boundary $\d M$. We assume that $\mathsf T(K,\e) \to K$ is a fibration whose fibers are the $(n-k)$-dimensional $\e$-balls, orthogonal to $K$. Thus the Weyl tube $\mathsf T(K,\e) \subset M$ is the union of all $\e$-balls, whose centers belong to $K$ (we assume that all $\e$ are such that $\mathsf{dist}(K, \d M) > \e$).

We do not know whether Hypothesis A generally holds. Therefore, we need to change the notion of a tube so that this hypothesis will become valid for the modified tubes. 

Let us adopt a definition from the paper of Burago, Ferleger, and  Kononenko \cite{BFK}.

\begin{definition}\label{def.nondegenerate} Let $\mathcal U$ be a finite collection of convex domains $\{B_i\}_{i \in [1, m]}$ in a smooth Riemmanian manifold $(M, g)$. Let $B :=  \bigcup_{i=1}^m B_i$. 

Then $\mathcal U$ is called {\sf non-degenerate in an open set} $U \subset M$, if for any multi-index $I \subset \{1, 2, \ldots, m\}$ and any $y \in (U \cap  B) \setminus \bigcap_{i \in I} B_i$,
$$\max_{k \in I}\; \Big\{\frac{\mathsf{dist}(y, B_k)}{\mathsf{dist}(y,\, \bigcap_{j \in I}B_j)}\Big\} \geq c,\; \text{ for a constant } c >0.$$
 
$\mathcal U$ is called {\sf non-degenerate}, if there exist $\delta >0$ and constant $c > 0$ so that $\mathcal U$ is non-degenerate with the constant $c >0$ in any $\delta$-ball. \hfill $\diamondsuit$
\end{definition}

Generally speaking, this means that if a point in $M \setminus U$ is $d$-close to all the walls $\{\d B_i\}_{i \in I}$ then it is $(d/c)$-close to their intersection.

\begin{definition}\label{def.bubbling} 
A {\sf tamed bubbling tube} $\mathbf T(K, \e)$ is a finite union $\bigcup_{j=1}^q B_\e(x_j)$ of closed $\e$-balls, whose distinct centers $x_j \in K$ and which satisfy the following properties: 
\begin{itemize}
\item $\mathbf T(K, \e) \cap K = K$,
\item the system of balls $\{B_\e(x_j)\}_{j \in [1, q]}$ is nondegenerate in the sense of Definition \ref{def.nondegenerate}, 
\item any two intersecting spheres $\d B_\e(x_j)$ and $\d B_\e(x_l)$ intersect transversally.
\end{itemize}

We call the number $\rho(\mathbf T(K, \e)) =_{\mathsf{def}} 1- \frac{vol(\mathbf T(K, \e))}{vol(\mathsf T(K, \e))} \in [0,1)$ the {\sf roughness} of the tamed bubbling tube $\mathbf T(K, \e)$. The bigger  $\rho(\mathbf T(K, \e))$ is, the ``rougher" is the tube. \smallskip
\hfill $\diamondsuit$
\end{definition}

Clearly, by adding new $\e$-balls with the centers in $K$, one can increase the volume  $vol(\mathbf T(K, \e)) \leq vol(\mathsf T(K, \e))$, while keeping new system of balls nondegenerate (see Fig.4). 

In what follows, we replace the original scattering problem for the manifold $N_\e := M \setminus \mathsf T(K, \e)$ with a similar scattering problem for $\mathbf N_{\e} =_{\mathsf{def}} M \setminus \mathbf T(K, \e)$. The main advantage of this replacement is that the billiard in $\mathbf N_{\e}$ is {\sf dispersing}, and therefore, has zero non-trapping volume (see Lemma \ref{lem.ZERO}). This conclusion is in line with the results of, i.g.,  \cite{BFK} which claim, in particular, that the trajectories of nondegenerate (see Definition \ref{def.nondegenerate}) semi-dispercing billiards have a finite number of collisions with the boundary, before escaping to infinity. 

At the same time, we can control the volume $vol(\mathbf T(K, \e))$, in terms of the intrinsic geometry of $K$, with accuracy that depends on the roughness $\rho(\mathbf T(K, \e))$.

Following \cite{BFK}, we declare any point of the intersection $\d B_\e(x_i) \cap \d B_\e(x_j)$, $i\neq j$, reached by a billiard trajectory $\g$ in $\mathbf N_{\e}$, to be \emph{the terminal point} on $\g$. Since such multiple intersections form a measure zero set in $\d \mathbf T(K, \e)$, the measure of the points in $S\mathbf N_{\e}$, visited by the lifts $\tilde \g$ of such $\g$, vanishes as well. Therefore, their contributions vanish in all the integral formulas to follow.

\begin{figure}[ht]
\centerline{\includegraphics[height=1.6in,width=3.5in]{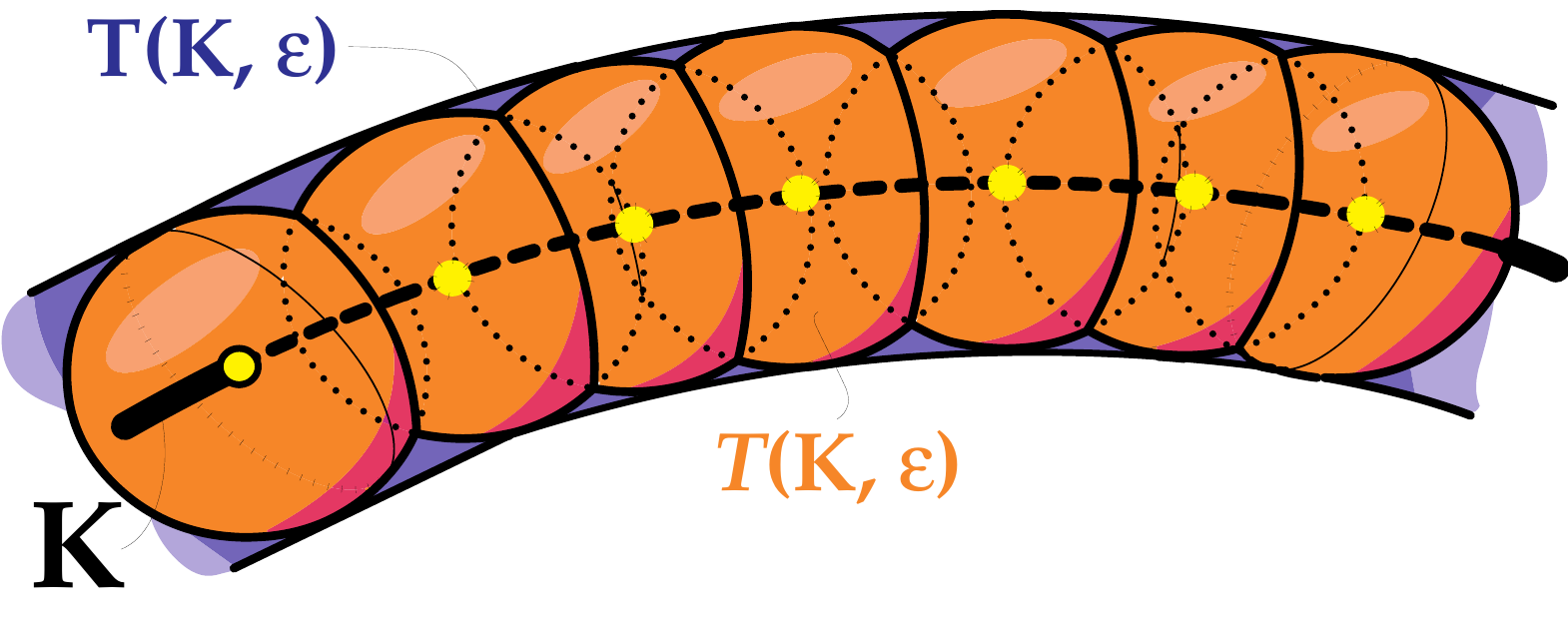}}\label{bill}
\bigskip
\caption{\small{Approximating a Weyl tube $\mathsf T(K, \e)$ by a tame bubbling tube $\mathbf T(K, \e)$.}} 
\end{figure}

\begin{lemma}\label{lem.ZERO} Let $\mathbf T(K, \e) \subset M$ be a tame bubbling tube. Consider  the billiard with the table $\mathbf N_\e := M \setminus \mathbf T(K, \e)$. Then the measure of the trapping set $\mathsf{Trap}(S\mathbf N_\e) \subset S\mathbf N_\e$ is zero.
\end{lemma}

\begin{proof} The set $\d\mathbf T(K, \e) \subset \E^n$ consists of pieces of $(n-1)$-spheres in $\E^n$. Each such piece is a manifold with corners. Indeed, it is possible to choose a tame bubbling tube so that any group of spheres intersect transversally (so that Definition \ref{def.nondegenerate} is satisfied).  

Without loss of generality, we may also assume that $\mathbf T(K, \e)$ is contained inside of a $n$-cube $Q \subset \E^n$, whose interior contains $M$. We view this cube as the fundamental domain for a flat $n$-torus $T^n$.  Thus, we may assume that $\mathbf T(K, \e)$ belongs to $T^n$. Consider now the billiard on the table $\mathbf P(K, \e) := T^n \setminus \mathsf{int}(\mathbf T(K, \e))$. 

The billiard $\mathbf P(K, \e)$ is {\sf dispersing}, i.e., each boundary component is convex inwards. Indeed, $\d\mathbf P(K, \e)$ consists of  pieces of spheres. However, it is not a Sinai billiard, because the boundary has singularities, which consists of the sets where at least two spheres intersect. Such singularities in the theory of billiards are called {\sf corner singularities}. Thanks to transversality of such intersections and to the fact that the number of spheres in $\d\mathbf T(K, \e)$ is finite, all the angles between the intersecting spheres exceed some value $\a >0$.

Actually, there are two types of singularities in our billiards $\mathbf P(K, \e)$. Only one type of singularities, $\Sigma^{\mathsf{tan}} \subset \d S\mathbf P(K, \e)$, the {tangent singularities}, is present in Sinai billiards, where the boundary is dispersing (convex inwards) and smooth. Note that $\dim \Sigma^{\mathsf{tan}} = 2n-3$. However, in billiards we consider, the boundary itself has the {\sf corner singularities} $\Sigma^{\vee} \subset \d S\mathbf P(K, \e)$, also of dimension $2n-3$. It is well known that ergodic properties of billiards are essentially determined by the time evolution of the singular set $\Sigma =_{\mathsf{def}} \Sigma^{\mathsf{tan}} \bigcup \Sigma^{\vee}$ under the iterations of the billiard map $\mathcal B := \mathcal B(K, \e): \d_1^+S\mathbf P(K, \e) \to \d_1^+S\mathbf P(K, \e)$. Thus, by the $\mathcal B$-invariant set $\bigcup_{i=-\infty}^{+\infty}\mathcal B^{(i)}(\Sigma)$ is everywhere dense in $\d_1^+S\mathbf P(K, \e)$ \cite{CM}.

Recall that Sinai billiards are ergodic \cite{Si}, \cite{KSS}.  For dispersing billiards,  
their ergodicity can be derived from, e.g., \cite{KSS}. We will make several comments in order to explain why  the billiards in the class $\mathbf P(K, \e)$ are ergodic. 

The main fact is that the images under positive and negative iterations of the billiard map $\mathcal B$ of any smooth piece of the singularity set $\Sigma$ do not locally coincide with 
any other piece of the singularity set \cite{CM}, \cite{SiC}, \cite{KSS}.  The last means that corresponding sets intersect transversally.
Indeed, consider the, so-called, {\sf wave fronts} (see \cite{CM}), i.e., local (small) manifolds, orthogonal to narrow beams of rays (billiard orbits). 

A smooth portion of the complement to the singularity set $\Sigma$ can be smoothly foliated by such strictly convex future wave  fronts (i.e., having a strictly positive curvature tensor).  
One just needs to perturb a unit velocity vector arbitrarily (the perturbation is again a unitary vector field), so that the foot points of the perturbed vectors remain the same.  In other words, we consider a beam of rays emanating from a point in a billiard table. 
Because of convexity of the boundary, these strictly convex wave fronts remain strictly convex under billiard dynamics \cite{CM}. 

Tamed bubbling tubes have strictly dispersing billiards with only corner-type singularities. In addition to that, at the corners the opening spatial angles between the spheres forming such tubes are always positive. There are also singularities arising because of tangencies of billiard orbits to the boundary of the tamed bubbling tubes. However, such singularities are not dangerous for our purposes \cite{BS}. Recall also that in dispersing billiards the set of orbits hitting (going through) singularities has zero measure.
 However, the corner singularities must be studied separately. One need to show that a future image of a singularity does not exactly coincide (not even locally) with a singularity.  The reason is that these singularities can be smoothly foliated with strictly convex future wave fronts, i. e. with strictly positive curvature tensor. The construction of these wave fronts is the so-called "candle" construction (\cite{KSS}).  One just needs to perturb only the unit velocity vector arbitrarily, but so that it stays unit, while the footpoint of the velocity vector stays fixed. Because the billiard is strictly dispersing, these strictly convex wave fronts stay strictly convex in the future.
Similarly, the backward images of singularities can be foliated by  the "inverse candles", i.e. strictly concave wave fronts with strictly negative definite curvature form.
These strictly convex and concave fronts are transversal to each other at any intersection point,
and ergodicity follows (\cite{SiC}).

The ergodicity of the billiard maps $\mathcal B(K, \e)$ follows, since, except for a set of zero measure and topological codimension at least two, there is a {\sf hyperbolic structure} in the complement to the singular set and its iterations in $\d_1^+S\mathbf P(K, \e)$, that is, local stable and unstable manifolds do exist almost everywhere. For such points, the conditions of the theorem on local ergodicity \cite{SiC} are satisfied. Recall that a dynamical system is called {\sf locally ergodic} if for almost every point there exists a neighborhood which belongs to one ergodic component.

A {\sf global ergodicity}, i.e., the existence of a single ergodic component of full measure, also follows. Indeed,  observe that the set of points, which satisfy the local ergodicity, is linearly connected and is at least of codimension two. 

Finally, contrary to the lemma claim, assume that the invariant set $A := \mathsf{Trap}(S\mathbf N_\e)$ has a positive measure. Then  $A$ is an invariant positive measure set in the spherical fibration $SQ \to Q$ over the cube $Q \supset M \supset \mathbf N_\e$. The set $A$ is invariant  under the factor map $S\E^n \to ST^n$ and clearly it is not of a full measure, since there are open families of lines through points of $M$ that miss $\mathbf P(K, \e) \subset \mathsf{int}(M)$.
On the other hand,  a strictly dispersing billiard $\mathbf P(K, \e)$ on a torus $T^n$ with scatterers $\d\mathbf T(K, \e)$ is ergodic \cite{Si}, \cite{SiC}. Thus we came to a contradiction, which proves  lemma. 
%
\end{proof}

\begin{remark}\emph{In our billiard $\mathbf P(K, \e)$, the scatters are pieces of spheres. Therefore, they are real semi-algebraic sets, and thus the results of  \cite{SiC}, \cite{KSS} apply directly.} \hfill $\diamondsuit$
 \end{remark}

\begin{theorem}\label{th.main_bubbling} Let $K$ be a closed smooth $k$-dimensional submanifold of the Euclidean space $\E^n$. Consider a sequence $\e_0 > \e_1 > \dots > \e_{\lceil k/2\rceil} > 0$, where $\e_0$ is sufficiently small. 
 Let $\{\mathbf T(K, \e_j)\}_j$ be tame bubbling tubes, contained in a smooth compact domain $M \subset \E^n$ so that $K$ is at least $\e_0$-away from $\d M$. 
 Consider the domains $\{\mathbf N_{\e_j} =_{\mathsf{def}} M \setminus \mathsf{int}(\mathbf T(K, \e_j))\}_j$. 
 
Then, the averages $\{\mathbf L^{\mathsf{av}}_{\e_j}\}_{0 \leq j \leq \lceil k/2\rceil}$ of lengths (travel times) functions  $$\big\{\mathbf L_{\e_j}: \d_1^+(SM) \to \R\big\}_{0 \leq j \leq \lceil k/2\rceil}$$ of the billiard trajectories (which originate and terminate in $\d M$) in the domains $\{\mathbf N_{\e_j}\}_{0 \leq j \leq \lceil k/2\rceil}$, together with the roughness coefficients $\{\rho_j:= \rho(\mathbf T(K, \e_j))\}_{0 \leq j \leq \lceil k/2\rceil}$ of the tame bubbling tubes, determine intrinsic global invariants $$\Big\{\mathsf Q_\ell := \int_K H_\ell \; \mu_K\Big\}_{0 \leq \ell \leq \lceil k/2 \rceil}$$ of the Riemannian manifold $K$, defined by the formulae (\ref{b})-(\ref{c}). 
\end{theorem}

\begin{proof} 
Since, by Lemma \ref{lem.ZERO}, the measure of $\mathsf{Trap}(S\mathbf N_{\e_j})$ is zero,  we get  $vol(\mathsf{Trap}(S\mathbf N_{\e_j})) = 0$. Therefore, 
$$\int_{\d_1^+(SM)} \mathbf L_{\e_j} \;\om_g^{n-1}\; = \int_{S\mathbf N_{\e_j}} \;\b_g \wedge \om_g^{n-1} =  vol(SM) - vol(S\mathbf T(K, \e_j))$$

By the definition of $\mathbf L^{\mathsf{av}}_{\e_j}$, we get 
$$\mathbf L^{\mathsf{av}}_{\e_j} \cdot \int_{\d_1^+(SM)}  \;\om_g^{n-1}\; 
= vol(SM) - vol(S\mathbf T(K, \e_j))$$
By the Fubini Theorem, the formula above transforms as
$$\mathbf L^{\mathsf{av}}_{\e_j} \cdot \theta_{n-1} \cdot vol(\d M)= \; \s_{n-1} \cdot \big[vol(M) - vol(\mathbf T(K,\e_j))\big],$$
where $\s_{n-1}$ denotes the volume of the unit $(n-1)$-sphere in $\E^n$ and $\theta_{n-1}$ the volume of the unit $(n-1)$-ball in $\E^{n-1}$.  

By the definition of $\rho_j := \rho(\mathbf T(K,\e_j))$ and putting $\hat\rho_j =_{\mathsf{def}}\, 1- \rho_j$,  
we get 
$$\mathbf L^{\mathsf{av}}_{\e_j} \cdot \theta_{n-1} \cdot vol(\d M)= \; \s_{n-1} \cdot \big[vol(M) - \hat\rho_j \cdot vol(\mathsf T(K,\e_j))\big].$$
Solving for $vol(\mathsf T(K,\e_j))$ leads to 
\begin{eqnarray}\label{eq.good}
vol(\mathsf T(K,\e_j)) = \hat\rho_j^{-1} \big\{vol(M) - [\theta_{n-1}/\s_{n-1}] \cdot vol(\d M) \cdot  \mathbf L^{\mathsf{av}}_{\e_j}\big\}.
\end{eqnarray}

According to (\ref{a}), $vol(\mathsf T(K,\e))$ is given by a polynomial $P_K(\e)$ of the form $\e^{n-k}\cdot Q_K(\e^2)$, where $Q_K(x)$ is a polynomial of degree $\lceil k/2 \rceil$ in $x$.
Therefore, using (\ref{eq.good}), 
\begin{eqnarray}\label{eq.observables}
 Q_K(\e_j^2) = \hat\rho_j^{-1} \e_j^{k-n}\big\{vol(M) - [\theta_{n-1}/\s_{n-1}] \cdot vol(\d M) \cdot  \mathbf L^{\mathsf{av}}_{\e_j}\big\}.
\end{eqnarray}
%

Note that for a relatively smooth bubbling tube $\mathbf T(K,\e_j)$ the ``smoothness" coefficient $\hat\rho^{-1}_j := \frac{vol(\mathbf T(K, \e))}{vol(\mathsf T(K, \e))}$ is close to $1$. Assuming  that all these roughness coefficients $\{\hat\rho_j\}_j$ are equal, i.e., $j$-independent, will simplify our computations.

Next, we fix a sequence $ \e_0 > \e_1 > \ldots > \e_{\lceil k/2 \rceil} > 0$ and apply the Lagrange interpolation formula 
$$Q_K(u^2) = \sum_{i=0}^{\lceil k/2 \rceil} \Big(\prod_{0\, \leq\, j \leq\, \lceil k/2 \rceil, \; j \neq i}\; \frac{u^2-\e^2_j}{\e^2_i - \e^2_j}\,\Big)\, Q_K(\e_i^2)$$
to reconstruct the polynomial 
$$Q_K(u^2) := \theta_{n-k} \times \; \sum_{\{\ell\, \in\, [0, k], \; \ell \equiv 0 \mod 2\}}\, \frac{\mathsf Q_\ell(K)}{(n-k+2)(n-k+4) \ldots (n-k +\ell)}\; u^{\ell}$$
 from the ``observables" $\{Q_K(\e_j^2)\}_j$ in (\ref{eq.observables}).


Now, we can conclude that the collection of numbers $\{\mathbf L^{\mathsf{av}}_{\e_j}\}_{0 \leq i \leq \lceil k/2 \rceil}$ determines a polynomial $Q_K(\e^2)$ and thus the polynomial $P_K = \e^{n-k}\cdot Q_K$, given by the Weyl formula (\ref{a}). As a result, the intrinsic integral invariants $\{\mathsf Q_\ell(K)\}$ in (\ref{b}) are determined by $\{\mathbf L^{\mathsf{av}}_{\e_j}\}_{0 \leq i \leq \lceil k/2 \rceil}$.
\end{proof}



\smallskip


\begin{conjecture} Let $K$ be a closed smooth $k$-dimensional submanifold of the flat torus $T^n$, $k \leq n-2)$.  Consider a  sequence of numbers $\{\e_0 > \e_1 > \dots > \e_{\lceil k/2\rceil} > 0\}$ such that   $\{\mathsf T(K, \e_j)\}_j$ are the Weyl tubes, properly contained in a smooth compact connected domain $M \subset T^n$. Assume that the natural homomorphism $\pi_1(M) \to \pi_1(T^n)$ of the fundamental groups is trivial \footnote{This assumption implies that the flat metric on $M$, {\bf induced by the metric on the ambient torus}, is non-trapping.} and the billiards $P(K, \e_j) =_{\mathsf{def}}\; T^n \setminus \mathsf{int}(\mathsf T(K, \e_j))$ are ergodic
 \smallskip
 
 Then the averages $\mathcal L^{\mathsf{av}}_{0 \leq j \leq \lceil k/2\rceil}$  of lengths (travel times) functions  $\big\{\mathcal L_{\e_j}: \d_1^+(SM) \to \R\big\}_{0 \leq j \leq \lceil k/2\rceil}$ of billiard trajectories (which originate and terminate in $\d M$) 
 in $\big\{N(K, \e_j) =_{\mathsf{def}} M \setminus \mathsf T(K, \e_j)\big\}_{0 \leq j \leq \lceil k/2\rceil}$ determine intrinsic global invariants $\Big\{\mathsf Q_\ell := \int_K H_\ell \; \mu_K\Big\}_{0 \leq \ell \leq \lceil k/2 \rceil}$ of the Riemannian manifold $K$, defined by the formulae (\ref{b})-(\ref{c}). \hfill $\diamondsuit$
\end{conjecture}

\section{Some concluding remarks and problems}


Weyl's Theorem \ref{We} admits versions for the volumes of tubes in the \emph{spherical} and \emph{projective} spaces, both real and complex. For  domains $M$ in these spaces, such that the metric on $M$ is non-trapping, many of our arguments seem to hold. 
 
In contrast, for compact domains $M$ in the hyperbolic spaces $\mathbb H^n$,  the Weyl equations for the volumes of $\e$-tubes turn into less informative \emph{inequalities} \cite{Gr}. As a result, for domains in $\mathbb H^n$, one may hope just to get some \emph{upper estimates} of the volume $vol(\mathsf T(K, \e))$ and of the trapping volume $vol(\mathsf{Trap}(SN_\e))$ in terms of the scattering data, as it is done in Theorem \ref{th.traping_volume}. Since by \cite{We} isometric embeddings of $K$ produce $\e$-tubes of the same volume, we conjecture that these estimates may be shared by all isometric embeddings $\a: K \hookrightarrow M \setminus \mathsf T(\d M, \e)$.\smallskip
 




\smallskip





%
The next problem seems to be quite challenging.\smallskip

\noindent{\bf Problem 6.1.} 
Let $M$ be a compact Riemannian manifold with boundary, equipped with a  non-trapping metric $g$.  Let $K$ be on a closed  Riemannian manifold which admits an isometric imbedding in $M$.  For a given (small) $\e > 0$, find the {\sf isometric} embeddings $\a_\star: K \hookrightarrow M$,  for which 
 the volume of the trapping set $\mathsf{Trap}(SN_\e(\a_\star(K)))$ attains its infimum/minimum.   

Clearly, the isometry group of $M$ acts on such optimal embeddings $\a_\star$. Thus, the optimal isometric embedding may be not unique.
\smallskip

Note that the problem of finding such optimal $\a_\star$'s is equivalent to the problem of finding an isometric embeddings $\a_\star: K \hookrightarrow M$ for which the average length of a  billiard trajectory $\mathcal L^{\mathsf{av}}\big(N_\e(\a(K))\big)$ in (\ref{eq.av_length}) attains a 
supremum. 
\hfill $\diamondsuit$

\smallskip


  
%

\smallskip
 
 {\it Acknowledgments:} We are indebted to N. Simanyi, D. Szasz and I. Toth for useful discussions.
The work of L.B. was partially supported by the NSF grant DMS-2054659.

We are also grateful to an anonymous referee for useful comments which allowed to improve the exposition.

\end{document}